\documentclass[12pt]{article}
\usepackage{a4wide}
\usepackage{amsmath,amsfonts,amssymb}
\usepackage{amsthm}
\usepackage{graphicx}

\usepackage{xcolor}

\usepackage{hyperref}
\usepackage{breakurl}

\definecolor{modra3}{rgb}{.1,.0,.4}
\hypersetup{colorlinks=true, linkcolor=modra3, urlcolor=modra3, citecolor=modra3, pdfpagemode=UseNone, pdfstartview=} 

\usepackage{orcidlink}

\newcommand{\lowerenv}[1]{\ensuremath{\text{low}(#1)}}
\makeatletter
\newcommand{\overleftrightsmallarrow}{\mathpalette{\overarrowsmall@\leftrightarrowfill@}}
\newcommand{\overrightsmallarrow}{\mathpalette{\overarrowsmall@\rightarrowfill@}}
\newcommand{\overleftsmallarrow}{\mathpalette{\overarrowsmall@\leftarrowfill@}}
\newcommand{\overarrowsmall@}[3]{%
  \vbox{%
    \ialign{%
      ##\crcr
      #1{\smaller@style{#2}}\crcr
      \noalign{\nointerlineskip}%
      $\m@th\hfil#2#3\hfil$\crcr
    }%
  }%
}
\def\smaller@style#1{%
  \ifx#1\displaystyle\scriptstyle\else
    \ifx#1\textstyle\scriptstyle\else
      \scriptscriptstyle
    \fi
  \fi
}
\makeatother

\newtheorem{theorem}{Theorem}        
\newtheorem{prop}[theorem]{Proposition}            
\newtheorem{lem}[theorem]{Lemma}
\newtheorem{cor}[theorem]{Corollary}

\theoremstyle{remark}
\newtheorem*{rem_}{Remark}

\begin{document}

\title{Extending simple monotone drawings\footnote{A preliminary version of this paper appeared in the proceedings of the 36th International Workshop on Combinatorial Algorithms (IWOCA 2025), \href{https://doi.org/10.1007/978-3-031-98740-3_2}{doi: 10.1007/978-3-031-98740-3\_2}.}
} 

\author{
Jan Kyn\v{c}l\footnote{Department of Applied Mathematics, Charles University, Faculty of Mathematics and Physics, Malostransk\'e n\'am.~25, 118 00~ Praha 1, Czech Republic; \texttt{\{kyncl,soukup\}@kam.mff.cuni.cz}. Supported by project 23-04949X of the Czech Science Foundation (GA\v{C}R) and by the grant SVV–2023–260699.  
}\ \,\orcidlink{0000-0003-4908-4703} 
\and 
Jan Soukup\footnotemark[2]\ \,\orcidlink{0000-0002-5039-3872}
} 

\date{}
\maketitle

\begin{abstract}
We prove the following variant of Levi's Enlargement Lemma: for an arbitrary arrangement $\mathcal{A}$ of $x$-monotone pseudosegments in the plane and a pair of points $a,b$ with distinct $x$-coordinates and not on the same pseudosegment, there exists a simple $x$-monotone curve with endpoints $a,b$ that intersects every curve of $\mathcal{A}$ at most once. As a consequence, every simple monotone drawing of a graph can be extended to a simple monotone drawing of a complete graph. We also show that extending an arrangement of cylindrically monotone pseudosegments is not always possible; in fact, the corresponding decision problem is NP-hard.
\end{abstract}

\date{}


\section{Introduction}
\label{section_intro}

Given $k\ge 1$, a finite set $\mathcal{A}$ of simple curves in the plane is called an \emph{arrangement of \mbox{$k$-strings}} if every pair of the curves of $\mathcal{A}$ intersects at most $k$ times, and every intersection point is either a proper crossing or an endpoint of at least one of the curves of the intersecting pair. Multiple $k$-strings can intersect at a~common point, an endpoint of one curve can lie on another, but two curves cannot touch in their inner points. An arrangement of $1$-strings is also called an \emph{arrangement of pseudosegments}, and each curve in the arrangement is called a \emph{pseudosegment}. In this paper, we represent simple curves as subsets of the plane (or another surface) that are homeomorphic images of a closed interval. When necessary, we also specify one endpoint of a curve as the starting point to introduce an orientation to the curve.

A simple curve $\gamma$ in the plane is \emph{$x$-monotone}, shortly \emph{monotone}, if $\gamma$ intersects every line parallel to the $y$-axis at most once.

Let $\mathcal{F}$ be a family of arrangements of pseudosegments in the plane, and let $a,b$ be a pair of points in the plane. An arrangement $\mathcal{A}$ from $\mathcal{F}$ is \emph{$(a,b)$-extendable in $\mathcal{F}$} if there exists a simple curve $\alpha$ with endpoints $a,b$ so that $\mathcal{A} \cup \{\alpha\} \in \mathcal{F}$. The arrangement $\mathcal{A}$ is \emph{extendable in $\mathcal{F}$} if it is $(a,b)$-extendable in $\mathcal{F}$ for all possible choices of $a$ and $b$ with distinct $x$-coordinates and not on the same pseudosegment. We omit the family $\mathcal{F}$ from the notation whenever it is clear from context.
The family $\mathcal{F}$ is \emph{extendable} if all its elements are extendable in $\mathcal{F}$.

Our main result is the following.

\begin{theorem}
\label{theorem_main}
The family of all arrangements of monotone pseudosegments in the plane is extendable.
\end{theorem}

Moreover, in Subsection~\ref{subsection_algorithm} we show that the proof of Theorem~\ref{theorem_main} can be turned into an efficient algorithm: given a suitable representation of an arrangement of $n$ monotone pseudosegments with $m$ incidences between endpoints or intersection points and pseudosegments, we can find the new pseudosegment extending the arrangement in time $O(m)$. 

A drawing of a graph in the plane is \emph{simple} if every pair of edges has at most one common point, either a common endpoint or a proper crossing. A drawing of a graph is \emph{monotone} if every edge is drawn as a monotone curve and no two vertices share the same $x$-coordinate. We have the following direct consequence of Theorem~\ref{theorem_main}, illustrated in Figure~\ref{fig_00_cor_2}.

\begin{cor}
\label{cor_drawings}
Every simple monotone drawing of a graph in the plane can be extended to a simple monotone drawing of the complete graph with the same set of vertices.
\end{cor}

\begin{figure}
\begin{center}		
\includegraphics{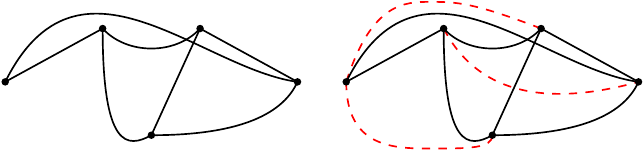}
\caption{Left: a simple monotone drawing of a graph. 
Right: an extension of the drawing on the left to a simple monotone drawing of a complete graph. The added edges are dashed.}
\label{fig_00_cor_2}
\end{center}		
\end{figure}

\subsection{Drawings on a cylinder}

A cylinder can be represented as the surface $S^1 \times \mathbb{R}$ embedded in $\mathbb{R}^3$, where $S^1$ is the unit circle in the $xy$-plane. A simple curve $\gamma$ on the cylinder is \emph{cylindrically monotone} if $\gamma$ intersects every line parallel to the $z$-axis at most once. A finite system of circular arcs on the circle $S^1$ is \emph{normal} if no pair of arcs covers the whole circle~\cite{LS09_circular}.
An arrangement of cylindrically monotone pseudosegments is \emph{normal} if the projections of its curves to $S^1$ form a normal system of circular arcs. Drawings of graphs whose edges are cylindrically monotone have also been called \emph{angularly monotone}~\cite{FR13_disjoint}, and in the case when the edges form a normal arrangement they have also been called \emph{strongly c-monotone}~\cite{AOV23_hamiltonian}.

Let $\mathcal{F}$ be a family of arrangements of pseudosegments on the cylinder, and let $a,b$ be a pair of points on the cylinder. An arrangement $\mathcal{A}$ from $\mathcal{F}$ is \emph{$(a,b)$-extendable in~$\mathcal{F}$} if there exists a simple curve $\alpha$ with endpoints $a,b$ so that $\mathcal{A} \cup \{\alpha\} \in \mathcal{F}$. The arrangement $\mathcal{A}$ is \emph{extendable in $\mathcal{F}$} if it is $(a,b)$-extendable in~$\mathcal{F}$ for all possible choices of $a$ and $b$ not on the same vertical line and not on the same pseudosegment. We omit the family $\mathcal{F}$ from the notation whenever it is clear from context.
The family $\mathcal{F}$ is \emph{extendable} if all its elements are extendable in $\mathcal{F}$.

Theorem~\ref{theorem_main} generalizes in a straightforward way to normal arrangements of cylindrically monotone pseudosegments:

\begin{cor}
\label{cor_cylinder}
The family of all normal arrangements of cylindrically monotone pseudosegments is extendable. 

\end{cor}

It may be helpful to consider an alternative definition of normal arrangements as those where the curves are drawn the ``shorter way'' around the cylinder. The following result implies that these two definitions are combinatorially equivalent.

\begin{prop}
\label{prop_normal}
Given a normal system $\mathcal{C}$ of circular arcs on the circle $S^1$, there exists a homeomorphism of $S^1$ that maps each arc in $\mathcal{C}$ to an arc of length smaller than $\pi$.
\end{prop}

We suspect that Proposition~\ref{prop_normal} might be a known result, but we were not able to find it in the literature.

We show that further generalization of Corollary~\ref{cor_cylinder} to arbitrary arrangements of cylindrically monotone pseudosegments is not possible.

\begin{prop}
\label{prop_obstruction}
There exists an arrangement $\mathcal{A}$ of five cylindrically monotone pseudosegments and a pair of points $a,b$ not on the pseudosegments of $\mathcal{A}$ and not on the same vertical line such that $\mathcal{A}$ is not $(a,b)$-extendable in the family of all arrangements of cylindrically monotone pseudosegments.
\end{prop}

Moreover, the decision problem of extendability of cylindrically monotone arrangements turns out to be NP-hard. 

\begin{theorem}
\label{theorem_hard}
Given an arrangement $\mathcal{A}$ of cylindrically monotone pseudosegments and a~pair of points $a,b$, it is NP-hard to decide whether $\mathcal{A}$ is $(a,b)$-extendable in the family of all arrangements of cylindrically monotone pseudosegments.
\end{theorem}

The problem in Theorem~\ref{theorem_hard} is, in fact, NP-complete: the membership in NP is rather straightforward since an arrangement of $n$ cylindrically monotone pseudosegments can be encoded by a sequence consisting of $O(n)$ endpoints and $O(n^2)$ crossings, and by the above/below relations between the pseudosegments and the endpoints or crossings of other pseudosegments.


We prove Theorem~\ref{theorem_main} in Section \ref{section_plane} and we prove Corollary \ref{cor_cylinder}, Proposition \ref{prop_normal}, Proposition \ref{prop_obstruction} and Theorem \ref{theorem_hard} in Section \ref{section_cylinder}. 

\subsection{Related results}

A \emph{pseudoline} in the plane is an image of a Euclidean line under a homeomorphism of the plane; in other words, a pseudoline is a homeomorphic image of the set $\mathbb{R}$, unbounded in both directions. An \emph{arrangement of pseudolines} is a finite set of pseudolines such that every pair of them has exactly one crossing, and no other common intersection point. Pseudolines are also often defined in the projective plane, as nonseparating simple closed curves.

Levi's Enlargement Lemma~\cite{L26_Teilung} states that for every arrangement of pseudolines and every pair of points $a,b$ not on the same pseudoline, one can draw a new pseudoline through $a$ and $b$, crossing every curve from the given arrangement exactly once. The lemma has several alternative proofs in the literature~\cite{AMRS18_Levi,S19_Levi}.

Snoeyink and Hershberger~\cite{SH91_sweeping} generalized Levi's Lemma to a sweeping theorem for pseudoline arrangements, which allows ``rotating'' a new pseudoline through a given point, sweeping the whole plane in the process.

By a classical result of Goodman~\cite{Go80_proof},~\cite[Theorem 5.1.4]{FG17_pseudoline}, every arrangement of pseudolines can be transformed by a homeomorphism of the plane into an~arrangement of monotone pseudolines, or a so-called \emph{wiring diagram}. Therefore, monotone arrangements of pseudosegments can be considered as a generalization of pseudoline arrangements. On the other hand, Figure~\ref{fig_00_non_pseudolines} shows an example that not every monotone arrangement of pseudosegments can be seen as a ``restriction'' of a pseudoline arrangement, and so Theorem~\ref{theorem_main} does not easily follow from Levi's Lemma. See Arroyo, Bensmail and Richter~\cite[Figure~2]{ABR21_pseudolines} for more examples. Since a pseudoline (in the projective plane) can be considered as a~union of two internally disjoint pseudosegments (one bounded and the other one crossing the line at infinity), Theorem~\ref{theorem_main} can also be considered as a generalization of ``a~half'' of Levi's Lemma.

\begin{figure}
\begin{center}		
\includegraphics{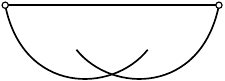}
\caption{An example of an arrangement of three pseudosegments that cannot be extended to pseudolines forming a pseudoline arrangement since any extension of the pseudosegments into pseudolines would contain a pair of pseudolines with two mutual intersections.}
\label{fig_00_non_pseudolines}
\end{center}		
\end{figure}

By $P_3+P_3$ we denote the union of two vertex-disjoint paths of length $2$. A simple drawing of $P_3+P_3$ that cannot be extended to a simple drawing of a complete graph was constructed by Eggleton~\cite[Diagram 15(ii)]{Egg73_crossing} 
and later rediscovered by the first author~\cite[Figure 9]{K13_improved}. 
Later a few more examples of non-extendable simple drawings were constructed~\cite[Figures 1, 10]{KPRT15_saturated}.
None of these drawings are homeomorphic to monotone drawings, which follows, for example, from Corollary~\ref{cor_drawings}, but some of them can easily be transformed into cylindrically monotone drawings; see Figure~\ref{figure_P3+P3}.

\begin{figure}
\begin{center}		
\includegraphics{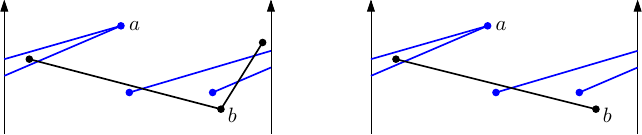}
\caption{Left: a simple cylindrically monotone drawing of $P_3+P_3$ where the edge~$ab$ cannot be added~\cite{Egg73_crossing,KPRT15_saturated,K13_improved} as a simple curve without crossing some edge of $P_3+P_3$. Since all edges of $P_3+P_3$ are incident to either $a$ or $b$, the added edge $ab$ would have at least two intersections with some edge of $P_3+P_3$.\\
Right: a simple cylindrically monotone drawing of $P_3+P_2$ where the edge $ab$ cannot be added as a cylindrically monotone curve without intersecting some edge of $P_3+P_2$ twice.}
\label{figure_P3+P3}
\end{center}		
\end{figure}

The first author together with Pach, Radoi\v{c}i\'{c} and T\'{o}th~\cite{KPRT15_saturated} proved that for every $k\ge 1$, for every arrangement $\mathcal{A}$ of $k$-strings and every pair of points $a,b$ not on the same curve of $\mathcal{A}$, there exists a simple curve joining $a$ and $b$ intersecting every curve of $\mathcal{A}$ at most $2k$ times. They also constructed examples showing that the constant $2k$ cannot be improved.

Arroyo, Klute, Parada, Vogtenhuber, Seidel and Wiedera~\cite{AKPVSW23_inserting} showed that it is NP-hard to decide, given an arrangement~$\mathcal{A}$ of pseudosegments and a pair of points $a,b$, whether $a$ and $b$ can be joined by a simple curve crossing each pseudosegment of $\mathcal{A}$ at most once. Our proof of Theorem~\ref{theorem_hard} is a simple adaptation of this result to cylindrically monotone arrangements. Recently, Aichholzer Orthaber and Vogtenhuber \cite{AOV24_Separable} proved a result similar to Theorem~\ref{theorem_main} where they showed that the crossing-minimizing and the so-called separable drawings can be extended to simple drawings of the complete graph but not necessarily to the crossing-minimizing or separable drawings of the complete graph.

Arroyo, Bensmail and Richter~\cite{ABR21_pseudolines} studied a slightly different extendability question: given an~arrangement of pseudosegments, can it be extended to an~arrangement of pseudolines by extending every given pseudosegment to a pseudoline? They determined the full infinite set of minimal obstructions, and found a polynomial time algorithm for detecting the obstructions and extending the arrangement.

\section{Monotone arrangements in the plane} \label{section_plane}

We start with a few definitions and tools for analyzing $x$-monotone arrangements. Given a pair of points $a,b$ in the plane, we write $a\prec b$ if $a$ has a smaller $x$-coordinate than $b$. Clearly,~$\prec$ is a strict linear order on the points of any monotone curve. 

We can naturally talk about objects lying ``below'' and ``above'' monotone curves. Let~$a,b$ be points such that $a\prec b$.  
For any monotone curve $\gamma$ we denote by $\gamma[a,b]$ and $\gamma(a,b)$ the subset of $\gamma$ formed by the points $x$ of $\gamma$ satisfying ${a\preceq x \preceq b}$ and $a\prec x \prec b$, respectively. Similarly, for an arrangement $\mathcal{A}$ of monotone pseudosegments we denote by $\mathcal{A}[a,b]$ the arrangement of pseudosegments where we replace each $\gamma\in \mathcal{A}$ by $\gamma[a,b]$. 

By \emph{consecutive intersections} of two monotone curves with finitely many intersections we mean consecutive intersections with respect to their $x$-coordinates.  
Let $\alpha, \beta$ be two monotone curves with finitely many intersections.
Let $a,b$ be two consecutive intersections of $\alpha, \beta$ such that $a \prec b$. Then the only intersections of $\alpha[a,b]$ with $\beta[a,b]$ are the points $a$ and $b$. In this case we say that the curves $\alpha$ and $\beta$ form a \emph{bigon}. Furthermore, if $\alpha(a,b)$ lies above $\beta(a,b)$ we say that $\alpha$ and $\beta$ form an \emph{$\alpha$-top}, or equivalently, a \emph{$\beta$-bottom} bigon. Note that in general $\alpha$ and $\beta$ can form both an $\alpha$-bottom and a $\beta$-bottom bigon.

The \emph{lower envelope} $\lowerenv{\mathcal{U}}$ of a set $\mathcal{U}$ of curves is the set of all points $p$ of these curves such that no other point of any curve of $\mathcal{U}$ with the same $x$-coordinate as $p$ is below $p$. Note that if $\mathcal{U}$ is an arrangement of monotone pseudosegments, then $\lowerenv{\mathcal{U}}$ is a finite union of connected parts of pseudosegments. 

We continue by proving Theorem~\ref{theorem_main}, followed by a discussion of the time complexity of extending the arrangement.

\subsection{Proof of Theorem~\ref{theorem_main}}

Let $\mathcal{A}$ be an arrangement of monotone pseudosegments.
Let $a,b$, with $a \prec b$, be points that are not on the same pseudosegment of $\mathcal{A}$. We need to find a~monotone curve from $a$ to $b$ that intersects every curve of $\mathcal{A}$ at most once. Since every curve of $\mathcal{A}$ is monotone, we can without loss of generality assume that $\mathcal{A} = \mathcal{A}[a,b]$.

\begin{figure}[tb]
\begin{center}		
	\includegraphics{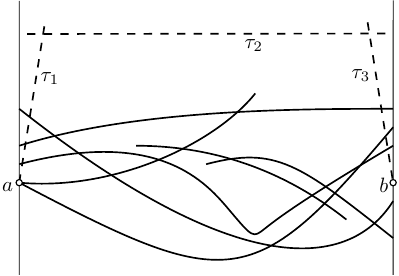}
	\caption{An arrangement of monotone pseudosegments with three added segments $\tau_1, \tau_2, \tau_3$ connecting points $a, b$ ``from above''.}
	\label{fig01_A_scarkou}
\end{center}		
\end{figure}

Let $\mathcal{A}'$ be an arrangement of monotone pseudosegments formed by all pseudosegments of $\mathcal{A}$ together with three new segments $\tau_1, \tau_2, \tau_3$, defined as follows. The segment $\tau_1$ is an almost vertical segment starting in $a$ and ending in some new point to the right of $a$ and above all pseudosegments of $\mathcal{A}$. Similarly, $\tau_3$ is an almost vertical segment ending in $b$ and starting in some new point to the left of~$b$ and above all pseudosegments of $\mathcal{A}$. Finally, $\tau_2$ is a horizontal segment crossing $\tau_1$ and $\tau_3$, and lying entirely above all pseudosegments of~$\mathcal{A}$; see Figure~\ref{fig01_A_scarkou}. 
In this way, $\lowerenv{\{\tau_1, \tau_2, \tau_3\}}$ is a monotone curve connecting $a$ and $b$ ``from above'', so that every pseudosegment $\gamma \in \mathcal{A}$ intersects it at most twice. Furthermore no $\gamma \in \mathcal{A}$ forms a $\gamma$-top bigon with $\lowerenv{\{\tau_1, \tau_2, \tau_3\}}$ (it can only form a $\gamma$-bottom bigon).

In order to find an extending curve we do the following. We find a nonempty subset $\mathcal{U}\subseteq\mathcal{A}'$ of pseudosegments such that the lower envelope of $\mathcal{U}$ is a monotone curve connecting $a$ to $b$, intersecting every pseudosegment of $\mathcal{A}'\setminus\mathcal{U}$ at most once. Furthermore, we find $\mathcal{U}$ so that no pseudosegment $\alpha$ touches $\lowerenv{\mathcal{U}}$ from below in an inner point of $\alpha$. After finding such~$\mathcal{U}$, a new pseudosegment connecting $a$ and $b$ can clearly be drawn slightly below the lower envelope of $\mathcal{U}$ and will indeed intersect every pseudosegment of $\mathcal{A}'$ at most once. Thus, if such $\mathcal{U}$ exists, $\mathcal{A}'$, and consequently $\mathcal{A}$, is $(a,b)$-extendable.

We find $\mathcal{U}$ inductively. We start with $\mathcal{U}_0 = \{\tau_1, \tau_2,  \tau_3 \}$ and always look at the lower envelope of $\mathcal{U}_i$. In the $i$th step we select an arbitrary pseudosegment $\gamma_i$ of $\mathcal{A}'\setminus \mathcal{U}_{i-1}$ intersecting $\lowerenv{\mathcal{U}_{i-1}}$ at least twice. If there is no such $\gamma_i$ then $\mathcal{U} = \mathcal{U}_{i-1}$ and we are done. Otherwise, we set $\mathcal{U}_{i}=\mathcal{U}_{i-1} \cup \{\gamma_i\}$. The number of pseudosegments is finite, so this process finishes with a set $\mathcal{U}$ such that the lower envelope of $\mathcal{U}$ intersects every pseudosegment of $\mathcal{A}'\setminus\mathcal{U}$ at most once.

Additionally, we prove that the induction preserves the following invariants for every~$\mathcal{U}_i$. 
\begin{enumerate}
\item[(I1)] No pseudosegment $\alpha$ of $\mathcal{A}' \setminus \mathcal{U}_{i}$ forms an $\alpha$-top bigon with $\lowerenv{\mathcal{U}_i}$.

\item[(I2)] No pseudosegment $\alpha$ of $\mathcal{A}' \setminus \mathcal{U}_{i}$ touches $\lowerenv{\mathcal{U}_i}$ from below in an inner point of $\alpha$.

\item[(I3)] The lower envelope of $\mathcal{U}_i$ is connected and contains $a$ and $b$.

\end{enumerate}

In particular, by (I3), the lower envelope of $\mathcal{U}$ is a monotone curve connecting $a$ to $b$ and, by (I2), no pseudosegment $\alpha$ of $\mathcal{A}' \setminus \mathcal{U}$ touches $\lowerenv{\mathcal{U}}$ from below in an inner point of $\alpha$. Since $\lowerenv{\mathcal{U}}$ intersects every pseudosegment of $\mathcal{A}'\setminus\mathcal{U}$ at most once by its construction, $\mathcal{A}$ is $(a,b)$-extendable by the previous discussion. Thus, it suffices to prove the correctness of these invariants to finish the proof.

The invariants hold for $\mathcal{U}_0$ by the construction of $\tau_1, \tau_2$ and $\tau_3$. Suppose all invariants hold for $\mathcal{U}_{i-1}$. 
In particular, $\lowerenv{\mathcal{U}_{i-1}}$ is a monotone curve connecting $a$ to $b$ by invariant~(I3).
We show that all invariants also hold for $\mathcal{U}_{i}$.

The pseudosegment $\gamma_i$ intersects $\lowerenv{\mathcal{U}_{i-1}}$ at least twice. We show that $\gamma_i$ intersects $\lowerenv{\mathcal{U}_{i-1}}$ exactly twice. Suppose, for contradiction, that there are three consecutive intersections $c, d$ and $e$ of $\gamma_i$ with $\lowerenv{\mathcal{U}_{i-1}}$ such that $c\prec d \prec e$. Then $\gamma_i[c,d]$ with $\lowerenv{\mathcal{U}_{i-1}}[c,d]$ forms a bigon and so does $\gamma_i[d,e]$ with $\lowerenv{\mathcal{U}_{i-1}}[d,e]$. By invariant (I1) both of these bigons must be $\lowerenv{\mathcal{U}_{i-1}}$-top bigons. However, in this case $\gamma_i$ touches $\lowerenv{\mathcal{U}_{i-1}}$ from below in the point $d$. That is not possible by invariant (I2). Thus, $\gamma_i$ intersects $\lowerenv{\mathcal{U}_{i-1}}$ exactly twice. Furthermore, by invariant (I1), $\gamma_i$ and $\lowerenv{\mathcal{U}_{i-1}}$ form a $\gamma_i$-bottom bigon. 

Let $x$ and $y$ be the two intersection points of $\gamma_i$ and $\lowerenv{\mathcal{U}_{i-1}}$. Refer to Figure~\ref{fig02_gamma_i}. Since $\gamma_i$ and $\lowerenv{\mathcal{U}_{i-1}}$ form a $\gamma_i$-bottom bigon, the only part of the curve $\gamma_i$ that lies below $\lowerenv{\mathcal{U}_{i-1}}$ is exactly $\gamma_i(x,y)$.
Thus, the lower envelope of $\mathcal{U}_{i-1} \cup \{\gamma_i\}$ is a monotone curve connecting $a$ and $b$. Therefore, invariant (I3) holds also for $\mathcal{U}_i$.

\begin{figure}
\begin{center}	
	\includegraphics{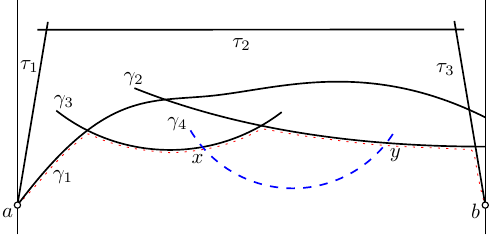}
	\caption{Induction step in the proof of Theorem~\ref{theorem_main}. In the $i$th step (fourth step in the figure) we add pseudosegment $\gamma_i$ (dashed) intersecting the lower envelope (dotted) of the previous segments twice. The lower envelope remains a connected curve connecting $a$ with $b$ even after this addition.} 
	\label{fig02_gamma_i}
\end{center}		
\end{figure}

Now, suppose that invariant (I2) does not hold, that is, there exists some pseudosegment~$\beta$ of $\mathcal{A}' \setminus \mathcal{U}_i$ that touches $\lowerenv{\mathcal{U}_i}$ from below in an inner point of~$\beta$. Refer to Figure~\ref{fig02.5_above_touch}. Since $\mathcal{U}_{i} = \mathcal{U}_{i-1} \cup \{\gamma_i\}$, the curve $\beta$ has to touch $\gamma_i$ or $\lowerenv{\mathcal{U}_{i-1}}$ in an inner point of $\beta$, a contradiction. Hence, invariant~(I2) also holds for $\mathcal{U}_i$. Note that the analogous statement for touchings from above does not hold, that is, there may exist some pseudosegment $\alpha$ of $\mathcal{A}' \setminus \mathcal{U}_i$ that both touches $\lowerenv{\mathcal{U}_i}$ from above in an inner point of $\alpha$ and touches neither $\gamma_i$ nor $\lowerenv{\mathcal{U}_{i}}$ in an inner point of $\alpha$.

\begin{figure}
\begin{center}	
	\includegraphics{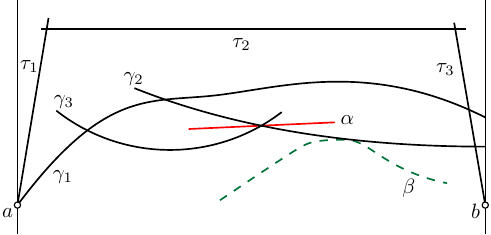}
	\caption{Induction step in the proof of Theorem~\ref{theorem_main}. During the selection of $\mathcal{U}$ some pseudosegments may touch $\lowerenv{\mathcal{U}}$ from above but never from below. Pseudosegment~$\alpha$ touches $\lowerenv{\{\tau_1, \tau_2, \tau_3, \gamma_1, \gamma_2 , \gamma_3\}}$ from above. On the other hand, $\beta$ cannot be in the same arrangement of pseudosegments since it touches $\gamma_2$.} 
	\label{fig02.5_above_touch}
\end{center}		
\end{figure}

Finally, suppose that invariant (I1) does not hold, that is, there exists some pseudosegment $\rho$ of $\mathcal{A}' \setminus \mathcal{U}_{i}$ that together with $\lowerenv{\mathcal{U}_{i}}$ forms a $\rho$-top bigon. Call $s$ and $t$ the vertices of this bigon and assume $s\prec t$. See Figure~\ref{fig03_rho}.

If $s$ and $t$ both lie on $\gamma_i[x,y]$, then $\rho$ and $\gamma_i$ intersect twice, a contradiction. Otherwise $s$ or $t$ does not lie on $\gamma_i[x,y]$. Without loss of generality assume that~$t$ does not lie on $\gamma_i[x,y]$ and $y\prec t$. Then $s$ either lies on $\lowerenv{\mathcal{U}_{i-1}}$ or below it. In both cases $\rho[s,t]$ intersects $\lowerenv{\mathcal{U}_{i-1}}$ in some point other than $t$ since $\rho[s,t]$ together with $\lowerenv{\mathcal{U}_{i}}$ forms a $\rho$-top bigon. Denote the rightmost intersection of $\rho[s,t]$ and $\lowerenv{\mathcal{U}_{i-1}}$ other than $t$ by $u$. Then $\rho(u,t)$ lies above $\lowerenv{\mathcal{U}_{i-1}}$ and so $\rho[u,t]$ together with $\lowerenv{\mathcal{U}_{i}}$ forms a $\rho$-top bigon, a contradiction with invariant~(I1) for $\mathcal{U}_{i-1}$. This concludes the proof of Theorem~\ref{theorem_main}.

\begin{figure}[tb]
\begin{center}	
	\includegraphics{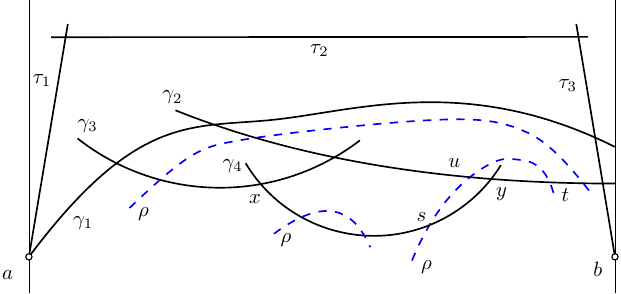}
	\caption{Induction step in the proof of Theorem~\ref{theorem_main}. If there was some pseudosegment $\rho$ that together with the lower envelope ($\lowerenv{\{\tau_1, \tau_2, \tau_3, \gamma_1, \gamma_2, \gamma_3, \gamma_4\}}$ in the picture) formed a $\rho$-top bigon, it would either form a $\rho$-top bigon with the previous lower envelope or intersect twice the segment that was added as the last. In the picture, there are three such possible $\rho$'s.} 
	\label{fig03_rho}
\end{center}		
\end{figure}

\subsection{Complexity of extending the arrangement}
\label{subsection_algorithm}

We now discuss how the proof of Theorem~\ref{theorem_main} can be turned into an algorithm. The outline of the algorithm can already be seen in the proof of Theorem~\ref{theorem_main}. Two main issues remain unresolved. Firstly, how to represent the arrangement, and secondly,  how to efficiently find a pseudosegment $\gamma_i$ of $\mathcal{A}'\setminus \mathcal{U}_{i-1}$ intersecting $\lowerenv{\mathcal{U}_{i-1}}$ at least twice. We explain what representation we can use and how to implement the algorithm in time linear in the number of incidences between pseudosegments and their endpoints or intersection points.

\subsubsection{Representation of a monotone arrangement}

We assume we have an arrangement $\mathcal{A}$ of $n$ pseudosegments. 
Combinatorially, the only points of interest on the pseudosegments are the endpoints of the pseudosegments and the intersection points of at least two pseudosegments. From now on, in the description of the algorithm, we will call these points of interest \emph{vertices}. Let $m$ be the total number of incidences between the pseudosegments and the vertices.

For every vertex $a$, we denote by $H^+_a$ and $H^-_a$  the open halfplane to the right and left, respectively, of the vertical line containing $a$. Let $x_a$ be the $x$-coordinate of $a$ and let $x_a^+$ and $x_a^-$ be $x$-coordinates that are slightly larger and slightly smaller, respectively, than $x_a$ so that there are no vertices with $x$-coordinates in the intervals $(x_a^-, x_a)$ and $(x_a, x_a^+)$.
For the input and output we use the following representation of arrangements of monotone pseudosegments: 

\begin{itemize}
    \item Every pseudosegment corresponds to a unique label
    \item Every vertex $a$ has its own structure containing its $x$- and $y$-coordinate and a list of incident pseudosegments together with the information whether $a$ is a left or right endpoint or an inner point of the pseudosegment. Furthermore, for every pseudosegment $\alpha$ that has $a$ as its left endpoint, the structure of $a$ contains the pair $(\alpha, f_a(\alpha))$ where $f_a(\alpha)$ is the label of the closest pseudosegment below $\alpha$ that intersects both $a$ and $H^+_a$ (or $f_a(\alpha)=\mathrm{NULL}$ if $\alpha$ is the bottommost pseudosegment with this property).
    If all pseudosegments incident with $a$ have $a$ as their left endpoint, the structure of $a$ contains the label $f(a)$ of the closest pseudosegment or vertex below $a$ (or $f(a)=\mathrm{NULL}$ if there is none). 
    \item The array $L$ of vertices is ordered in ascending order first by their $x$-coordinates and then by their $y$-coordinates.
\end{itemize} 

The main operation we use on this representation is \emph{Sweep}, during which we sweep the plane from left to right and maintain the vertical order of pseudosegments in linear time. We do it by processing the vertices in $L$ in the given order. At each vertex $a$ of $L$ we reverse the vertical order of all pseudosegments that contain $a$ as an inner point, we delete all pseudosegments that contain $a$ as their right endpoint, and, finally, we insert each pseudosegment $\alpha$ that has $a$ as its left endpoint to an appropriate position in the order determined by $f_a(\alpha)$ and $f(a)$. We process all vertices with the same $x$-coordinate together. For any vertex $a$, we compute the vertical order at coordinates $x_a^-$ and $x_a^+$, not at the point $a$ itself. 
Overall, the procedure takes time $O(m)$. 

Let us now discuss several aspects of this representation. The coordinates of the vertices are not strictly necessary. Instead, it would be sufficient to have a partition of the set of vertices into several classes, a linear ordering of the classes, and a linear ordering of each class. Here each class of the partition would represent a set of vertices with the same $x$-coordinate.
Note that we cannot assume that no two vertices have the same $x$-coordinate because if we tried to ``rotate'' the arrangement a bit to avoid this, in the rotated arrangement we could weave a monotone pseudosegment in between the two rotated points, but the corresponding pseudosegment in the original arrangement would not be monotone. Alternatively, the lexicographic order of $L$ is not strictly needed if the coordinates of the vertices are given. In this case, we could compute the ordering in time $O(m\log(m))$.

\subsubsection{Initialization and auxiliary structures}

In the input, we need to know the position of the points $a$ and $b$. We assume we are given their $(x,y)$-coordinates and also which vertices or pseudosegments are right below $a$ and $b$. We also assume that $\mathcal{A} = \mathcal{A}[a,b]$, otherwise we would first shorten the arrangement. 

Firstly, we construct $\mathcal{A}'=\mathcal{A}\cup\{\tau_1,\tau_2,\tau_3\}$ from $\mathcal{A}$. To add $\tau_1$, we first create a new vertex $a$ if $a$ was not a vertex before. That is straightforward, because in that case there are no incidences with $a$ and to $f(a)$ we assign the vertex or pseudosegment that is right below $a$, which is a part of the input. Next, we compute the vertical order of pseudosegments at the coordinate $x_a^+$. The new pseudosegment $\tau_1$ then intersects all pseudosegments in this order above $a$. It is straightforward to add all these intersections and the new pseudosegment to the representation. We add $\tau_2$ and $\tau_3$ similarly.

We now create a more detailed representation of the arrangement that will help us describe an efficient algorithm. We create the following data structures.

\begin{itemize}
	\item We define two types of incidences between vertices and pseudosegments. A ``$+$''-incidence $(p, \alpha)^+$ is an incidence between a vertex $p$ and a pseudosegment $\alpha$ that intersect $H^+_p$. Similarly, a ``$-$''-incidence $(p, \alpha)^-$ is an incidence between a vertex $p$ and a pseudosegment $\alpha$ that intersect $H^-_p$. We note that if $p$ is an interior point of $\alpha$, then the incidence between $p$ and $\alpha$ is of both types, ``$+$'' and ``$-$''. We define the (partial) \emph{horizontal order} on these incidences first by comparing the $x$-coordinates of their vertices, and in addition, for every pair of incidences $(p,\alpha)^-$ and $(p,\beta)^+$ sharing the same vertex $p$, the ``$-$''-incidence comes before the ``$+$''-incidence. 
    
	\item For every pseudosegment $\alpha$ we create an array of ``$+$'' and ``$-$''-incidences containing $\alpha$, linearly ordered according to their horizontal order. For example, a pseudosegment $\beta$ starting in $p$, ending in $q$ and containing one more vertex $r$ in the middle is represented by the array $[(p,\alpha)^+, (r,\alpha)^-, (r,\alpha)^+, (q,\alpha)^-]$.
	
	\item  We use the original structure for vertices but we add two lists into it. For each vertex $p$ we add a list $V^-(p)$ containing pointers to all the ``$-$''-incidences between $p$ and pseudosegments ordered from top to bottom at the $x_p^-$ coordinate. Similarly, we add a list $V^+(p)$ containing pointers to all the ``$+$''-incidences between $p$ and pseudosegments ordered from top to bottom at the $x_p^+$ coordinate.  To each incidence of either form we also add pointers to its location in $V^+(p)$ and $V^-(p)$.

\end{itemize}
This can be done by a sweep from left to right, so in time $O(m)$.

Note that given two incidences between a point $p$ and two pseudosegments $\alpha, \beta$ incident with $p$, we can access any cyclical interval of pseudosegments around $p$ in between $\alpha$ and $\beta$ in time linear in the size of the interval. 

The main part that remains is to describe how to represent the lower envelope in each step of the algorithm, how to select the pseudosegment $\gamma_i$ effectively, and then how to update the lower envelope. We maintain the following auxiliary structures.
\begin{itemize}
    
    \item A linked list $X$ of pointers to the ``$+$'' and ``$-$''-incidences on $\lowerenv{\mathcal{U}_{i-1}}$ linearly ordered according to their horizontal order. To each vertex $p$ we add pointers to the incidences in this list that contain $p$.

    \item A linked list $Y$ that for each pseudosegment $\alpha$ from $\mathcal{A}' \setminus \mathcal{U}_{i-1}$ contains the list of all intersections between $\lowerenv{\mathcal{U}_{i-1}}$ and $\alpha$ that are either proper intersections or endpoints of $\alpha$ where $\alpha$ continues below $\lowerenv{\mathcal{U}_{i-1}}$. From the proof of Theorem~\ref{theorem_main}, we know that any item from $Y$ contains at most two intersections. Together with these intersections, we also store the pointers to their positions in $X$. For each pseudosegment, we also include a pointer to this list.
    
    \item A set of pseudosegments $Z$ from $\mathcal{A}' \setminus \mathcal{U}_{i-1}$ intersecting the lower envelope of $\mathcal{U}_{i-1}$ twice.
    
    \item For each $\alpha \in \mathcal{A}'$ we remember whether it is in $\mathcal{U}_{i-1}$.
\end{itemize}

These structures can be constructed for $\mathcal{U}_0$ in time $O(n)$ because points $a$ and $b$ are incident to at most $n$ pseudosegments and all of the other at most $O(n)$ vertices on $\tau_1, \tau_2, \tau_3$ are simple intersections of only two pseudosegments.

\subsubsection{Implementation of the inductive step}

In each inductive step, we select $\gamma_i$ as an arbitrary pseudosegment from $Z$ and delete it from $Z$. Let $x$ and $y$ be its intersections with the lower envelope of $\mathcal{U}_{i-1}$ (we get them from $Y$). Let $(x,\alpha)^+$ be the ``$+$''-incidence between $x$ and $\lowerenv{\mathcal{U}_{i-1}}$ and $(y,\beta)^-$ be the ``$-$''-incidence between $y$ and $\lowerenv{\mathcal{U}_{i-1}}$.  
We go along the boundary of the bigon formed by $\lowerenv{\mathcal{U}_{i-1}}$ and $\gamma_i$ and update the data structures $X,Y$ and $Z$ for $\mathcal{U}_{i-1}$ to  $X,Y$ and $Z$ for $\mathcal{U}_{i}$. The implementation is straightforward using the described data structures: we delete the intersection information of $\lowerenv{\mathcal{U}_{i-1}[x,y]}$ with incident pseudosegments below it and add the intersection information of $\gamma_i[x,y]$ with incident pseudosegment below it. At point $x$ we need to delete the intersection information with pseudosegments from the incidences from $V^+(x)$ between $(x,\alpha)^+$ and $(x,\gamma_i)^+$. At point $y$ we need to delete the intersection information with pseudosegments from the incidences from $V^-(y)$ between $(y,\beta)^-$ and $(y,\gamma_i)^-$.
As noted before we can access any cyclical interval of pseudosegments around any vertex in time linear in the size of the interval. Furthermore, all bigons in the algorithm are disjoint, so during the algorithm we spend at most a constant amount of time for each incidence between a pseudosegment and a vertex.

At the end we get the list $X$ describing the lower envelope of $\mathcal{U}$ and the list $Y$ that contains all the the intersection of $\lowerenv{\mathcal{U}}$ with pseudosegments that continue below $\lowerenv{\mathcal{U}}$.

\subsubsection{Adding the new pseudosegment into the arrangement}
We have the linked lists $X$ and $Y$ describing $\lowerenv{\mathcal{U}}$ and its intersections with $\mathcal{A}' \setminus \mathcal{U}$ that continue below $\lowerenv{\mathcal{U}}$. The new pseudosegment lies slightly below $\lowerenv{\mathcal{U}}$. Thus, it intersect exactly the same pseudosegments that are described in $Y$. The intersections with the new pseudosegment will be slightly below and slightly to the left or right of the intersections with $\lowerenv{\mathcal{U}}$. Let $p$ be an old intersection of $\lowerenv{\mathcal{U}}$ with a pseudosegment $\gamma$ from $Y$ with ``$+$'' and ``$-$''-incidences $(p,\alpha)^-$ and $(p,\beta)^+$ on $\lowerenv{\mathcal{U}}$. To decide whether the new intersection $p'$ below $p$ should be to the right or to the left of $p$, we just need to find $\gamma$ either below $\alpha$ in $V(p)^-$ or below $\beta$ in $V(p)^+$. Given all this information it is straightforward to add the new pseudosegment into the original representation of $\mathcal{A}$ to update the arrangement.

\section{Cylindrically monotone arrangements}  \label{section_cylinder}

Recall that we represent the cylinder as the surface $S^1 \times \mathbb{R}$. If we imagine cutting and unrolling the cylinder, we can represent it in the plane as a vertical strip whose left and right sides represent the same vertical line of the cylinder. We can also select the orientation so that the counter-clockwise direction on the cylinder corresponds to the left-to-right direction in the plane. For points $a,b \in S^1$ we denote by $[a,b]$ the counter-clockwise circular arc in $S^1$ from $a$ to $b$ and we call it an \emph{interval}. For points $c$ and $d$ on the cylinder we call a monotone curve starting in $c$ and ending in $d$ on the cylinder \emph{left-oriented} or \emph{right-oriented} if its projection by the canonical projection map $\pi$ to $S^1$ is a clockwise circular arc or a counter-clockwise circular arc, respectively. 

\subsection{Normal cylindrically monotone arrangements}

We prove Corollary~\ref{cor_cylinder} using the alternative characterization of normal cylindrically monotone arrangements provided by Proposition~\ref{prop_normal}. Then we prove the proposition itself.

\subsubsection{Proof of Corollary \ref{cor_cylinder}}
Let $\mathcal{A}$ be a normal arrangement of cylindrically monotone pseudosegments. Let $a,b$ be a pair of points on the cylinder but not on the same vertical line and not on the same pseudosegment of $\mathcal{A}$. We need to show that $\mathcal{A}$ is $(a,b)$-extendable in the family of all normal arrangements of cylindrically monotone pseudosegments. That is, we need to find a cylindrically monotone simple curve $\gamma$ with endpoints $a, b$ that intersects every curve of $\mathcal{A}$ at most once such that $\mathcal{A} \cup \{\gamma\}$ is still normal.

Let $\pi$ be the orthogonal projection from $S^1 \times \mathbb{R}$ to $S^1$. Let $f$ be the homeomorphism from Proposition \ref{prop_normal}. We define a homeomorphism $g\colon S^1 \times \mathbb{R} \rightarrow S^1 \times \mathbb{R}$ as $g(x,y) = (f(x),y)$. The image $g(\mathcal{A}) = \left\{g[\alpha];\,\alpha \in \mathcal{A}\right\}$ of $\mathcal{A}$ under $g$ is a homeomorphic arrangement of cylindrically monotone pseudosegments whose orthogonal projections to $S^1$ have lengths smaller than $\pi$. Furthermore, $g$ maps every vertical line onto a vertical line. Thus, extending $\mathcal{A}$ by a pseudosegment from $a$ to $b$ is equivalent to extending $g(\mathcal{A})$ by a pseudosegment from $g(a)$ to $g(b)$.

We look at the orthogonal projections $\pi(g(a)), \pi(g(b))$ of $g(a), g(b)$ to $S^1$. Without loss of generality we may assume that the length of $\left[\pi(g(a)),\pi(g(b))\right]$ is smaller than $\pi$. Since orthogonal projections of all pseudosegments of $g(A)$ to $S^1$ have lengths smaller than $\pi$, then
\[
\left\{ g[\alpha] \cap \left(\left[\pi(g(a)),\pi(g(b))\right] \times \mathbb{R}\right);\,\alpha \in \mathcal{A}\right\} \setminus \{\emptyset\}
\]
is an arrangement of cylindrically monotone pseudosegments on a cylindrical strip. This arrangement is homeomorphic to an arrangement of monotone pseudosegments in the plane (by unrolling the strip to the plane). Thus, by Theorem~\ref{theorem_main}, we can extend the arrangement in the plane by a monotone pseudosegment $\beta$ connecting the images of $g(a)$ and $g(b)$ after unrolling. Let $\beta'$ be the cylindrically monotone curve on the cylinder corresponding to $\beta$. Clearly, $\beta'$ is a curve connecting $g(a)$ and $g(b)$ and intersecting every pseudosegment from $g(\mathcal{A})$ at most once. Moreover, since the length of $\left[\pi(g(a)),\pi(g(b))\right]$ is smaller than $\pi$, the resulting arrangement is still normal. Thus, $\gamma = g^{-1}(\beta')$ is the desired curve extending the original arrangement $\mathcal{A}$.

\begin{rem_}
Let us note that we could also find the extending curve $\gamma$ without the help of Proposition~\ref{prop_normal} by ``cutting'' the cylinder along the vertical lines passing through $a$ and $b$ and showing that one of the resulting parts corresponds to the planar version.
\end{rem_}


\subsubsection{Proof of Proposition~\ref{prop_normal}}
Let $\mathcal{C}$ be a normal system of $n$ arcs in the unit circle $S^1$. Without loss of generality, we assume that their endpoints are pairwise distinct. 
Let $V=\{v_1,v_2,\dots, v_{2n}\}$ be the set of the $2n$ endpoints of the arcs from $\mathcal{C}$, labeled arbitrarily.

Our first goal is to define, for each endpoint $v\in V$, its ``antipodal'' point $v'$, so that these new $2n$ points, together with the original $2n$ endpoints from $V$, are pairwise distinct, no arc from $\mathcal{C}$ contains both $v$ and $v'$, and for every pair $v,w\in V$ with $v\neq w$, the interval $[v,v']$ contains exactly one point from $\{w,w'\}$; that is, the endpoints of the intervals $[v,v']$ and $[w,w']$ alternate on the circle. 

Let $\mathcal{C}_0=\mathcal{C}$. In $2n$ steps, we will define a sequence of points $v'_1, v'_2, \dots, v'_{2n}$ and a sequence of sets $\mathcal{C}_1, \mathcal{C}_2, \dots, \mathcal{C}_{2n}$. We will make sure the following properties are satisfied:
\begin{enumerate}
\item[(P1)] Each $\mathcal{C}_i$ is a set of $n+2i$ intervals; in particular, $\mathcal{C}_i=\mathcal{C} \cup \bigcup_{j=1}^i \{[v_j,v'_j),[v'_j,v_j)\}$.
\item[(P2)] The points $v_1,v_2,\dots,v_{2n},v'_1,v'_2,\dots,v'_i$ are pairwise distinct.
\item[(P3)] No pair of the intervals from $\mathcal{C}_i$ with nonempty intersection covers the whole circle~$S^1$.
\end{enumerate}

By the assumption, all three properties (P1)--(P3) are satisfied for $\mathcal{C}_0$.

Let $i\in \{1,\dots,2n\}$. Assume that we have defined $\mathcal{C}_{i-1}$ so that (P1)--(P3) are satisfied for $\mathcal{C}_{i-1}$.

Let 
\[
J_i=\bigcup \{I\in\mathcal{C}_{i-1}; v_i\in I\} \ \text{ and }\ J'_i=\bigcap \{S^1\setminus I; I\in\mathcal{C}_{i-1} \wedge v_i\in I\}.
\]
That is, $J_i$ is the union of all intervals from $\mathcal{C}_{i-1}$ that contain $v_i$, and $J'_i$ is the complement~$S^1\setminus J_i$. We claim that $J'_i$ is an interval of positive length. Indeed, by property (P3) and by the $1$-dimensional Helly theorem applied to the intervals $S^1\setminus I$ where $I\in \mathcal{C}_{i-1}$ and $v_i\in I$, these intervals intersect at a common point. In fact, by properties (P1) and (P2), the endpoints of these intervals are pairwise disjoint, and so their common intersection $J'_i$ is a nontrivial interval.

Let $v'_i$ be an arbitary point from the interior of $J'_i$ that is distinct from all the points $v_1,v_2,\dots,\allowbreak v_{2n},v'_1,v'_2,\dots,v'_{i-1}$, and let $\mathcal{C}_i=\mathcal{C}_{i-1} \cup \{[v_i,v'_i),[v'_i,v_i)\}$. Clearly, $\mathcal{C}_i$ satisfies properties (P1) and (P2) by construction. Now we verify property (P3). Suppose that one of the intervals $[v_i,v'_i),[v'_i,v_i)$, together with some interval $I\in\mathcal{C}_{i-1}$, cover the whole circle. If $v_i\in I$, then $I\subseteq J_i$ and none of the intervals $[v_i,v'_i),[v'_i,v_i)$ covers the complement $J'_i$ by the choice of $v'_i$. If $v_i\notin I$, then some neighborhood $N(v_i)$ of $v_i$ is disjoint with $I$, but none of the intervals $[v_i,v'_i),[v'_i,v_i)$ contains $N(v_i)$. Therefore, property (P3) is satisfied also for~$\mathcal{C}_i$.

It remains to define the desired homeomorphism of the circle. Let $w_1,w_2,\dots,\allowbreak w_{4n}$ be a relabeling of the points $v_1,v_2,\dots,v_{2n},v'_1,v'_2,\dots,v'_{2n}$ in the counter-clockwise cyclic order around the circle. By property (P3), every pair of points $\{v_i,v'_i\}$ is relabeled as a pair $\{w_j,w_{j+2n}\}$ for some $j$.
We will map the $4n$ points $w_j$ to the vertices of a regular $4n$-gon, keeping their circular ordering.

Define the homeomorphism $f:S^1 \rightarrow S^1$ as follows. For every $j\in\{1,2,\dots,\allowbreak 4n\}$, let 
$f(w_j)=e^{2\pi i j/(4n)}$, 
and interpolate $f$ in each of the intervals $[w_j,w_{j+1}]$ by an arbitrary homeomorphism with the interval $[f(w_j),f(w_{j+1})]$. By the previous observation, each pair~$\{v_i,v'_i\}$ is mapped by $f$ to a pair of antipodal points on the circle. Therefore, by property (P3), every 
arc from $\mathcal{C}$ is mapped by $f$ to an interval of length at most $\pi \cdot (1-1/(2n))$.

\subsection{General cylindrically monotone arrangements}
In this subsection, when we say that an arrangement is extendable, we mean that it is extendable in the family of all cylindrically monotone arrangements. 

Unlike normal arrangements, general arrangements of cylindrically monotone pseudosegments are not always extendable, as stated in Proposition~\ref{prop_obstruction}.
Moreover, Theorem~\ref{theorem_hard} states that in general, deciding whether a cylindrically monotone arrangement is $(a,b)$-extendable is NP-hard.

\subsubsection{Proof of Proposition~\ref{prop_obstruction}}
Figure~\ref{fig04_nonextendable} provides an example of an arrangement $\mathcal{A}$ of five pseudosegments and a pair of points $a$, $b$ such that every cylindrically monotone curve connecting $a$ and $b$ has to intersect some pseudosegment of $\mathcal{A}$ twice. This means that $\mathcal{A}$ is not $(a,b)$-extendable. This example is the smallest that we know of where the points~$a$ and $b$ do not lie on any of the curves of~$\mathcal{A}$.

\begin{figure}
\begin{center}	
\includegraphics[scale=1]{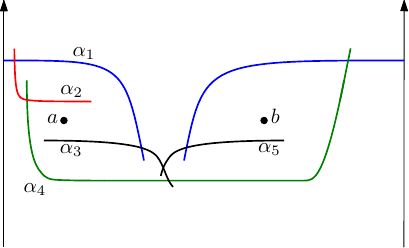}
\caption{An arrangement of five cylindrically monotone pseudosegments that is not $(a,b)$-extendable.} 
\label{fig04_nonextendable}
\end{center}		
\end{figure}

\subsubsection{Proof of Theorem~\ref{theorem_hard}}
 Our proof is heavily inspired by a recent proof by Arroyo et al.~\cite{AKPVSW23_inserting} showing that deciding the extendability of a simple drawing of a graph by one edge between a prescribed pair of points is NP-hard. We adapt their reduction to the more restricted setting of cylindrically monotone drawings.

We prove the NP-hardness by a reduction from 3-SAT. Let $\phi(x_1,\dots, x_n)$ be a
3-SAT-formula with variables $x_1,\dots, x_n$ and a set of clauses $C = \{C_1,\dots,C_m\}$. A \emph{literal} is an occurrence of a variable in a clause. 

We use the following lemma by Arroyo et al.~\cite{AKPVSW23_inserting}.
\begin{lem}[\cite{AKPVSW23_inserting}]
    The following transformation of a clause with only positive or only negative literals, respectively, preserves the satisfiability of the clause
	($y$ is a new variable and $\mathtt{false_p}$, $\mathtt{false_n}$ are constants with value false): 
\begin{align*}
		x_i \! \lor \! x_j \! \lor \! x_k 
        \ \ \longrightarrow & \ \ \ 
		\big( 
			x_k \! \lor \!  y \lor \mathtt{false_n}   
            \big) \wedge
            \big( 
			x_i \! \lor \!   x_j \! \lor \! \neg y   
		 \big)\\
		\neg x_i \! \lor \! \neg x_j \! \lor \! \neg x_k 
        \ \ \longrightarrow & \ \ \ 
		\big( 
			\neg x_i  \! \lor  \! \neg x_j  \! \lor  \! y \big) \wedge
			\big( 
            \neg x_k  \! \lor  \! \neg y  \! \lor  \! \mathtt{false_p} \big)
	\end{align*}
\end{lem}
We use the transformation from this lemma on every clause containing only positive or only negative literals. We obtain a new 3-SAT-formula with the same satisfiability and a modified set of clauses and variables (and new constants $\mathtt{false_p}$, $\mathtt{false_n}$). For simplicity, we abuse the notation and use the same notation for the new formulas, variables, and clauses. Furthermore, we consider each occurrence of $\mathtt{false_p}$ to be a positive literal and each occurrence of $\mathtt{false_n}$ to be a negative literal. The new formula $\phi(x_1,\dots, x_n)$ has exactly three literals in each clause and no clause has only positive or only negative literals. 

We construct an arrangement $\mathcal{A}$ of cylindrically monotone pseudosegments and a pair of points $a,b$, 
such that $\mathcal{A}$ is $(a,b)$-extendable if on only if $\phi(x_1,\dots, x_n)$ is satisfiable. Since 3-SAT is NP-hard the theorem will follow.

\begin{figure}
\begin{center}	
\includegraphics[scale=1]{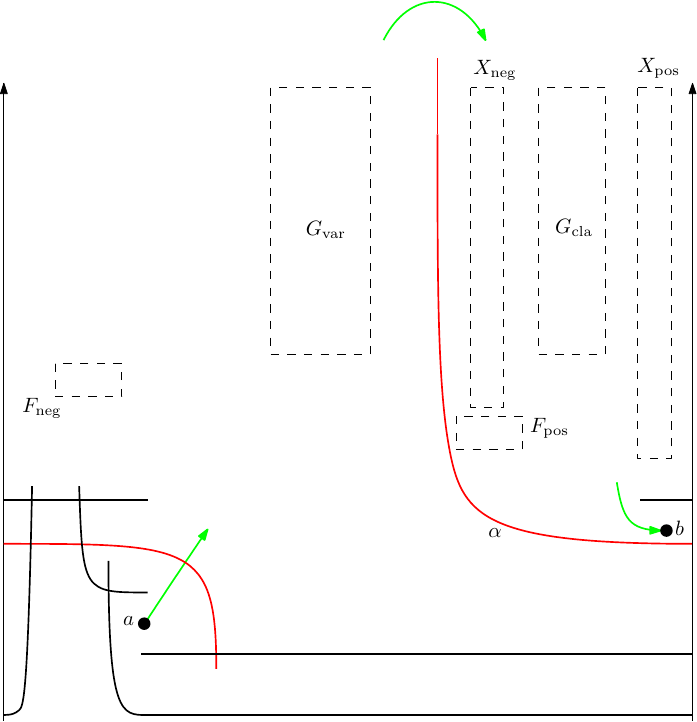}
\caption{An arrangement $\mathcal{A}_1$ of cylindrically monotone pseudosegments (depicted by solid black and red lines). Any monotone curve joining $a$ and $b$ and extending $A_1$ would have to be drawn in a way suggested by the (green) arrows. Regions~$F_{\mathrm{neg}}, F_{\mathrm{pos}}, G_{\mathrm{var}}, G_{\mathrm{cla}}$ illustrate the approximate positioning of gadgets from the proof: negative constant gadgets, positive constant gadgets, variable gadgets, and clause gadgets, respectively. Regions $X_{\mathrm{neg}}$ and $X_{\mathrm{pos}}$ illustrate the only areas where pseudosegments corresponding to the negative and positive literals cross, respectively.} 
\label{fig05_NP-start}
\end{center}		
\end{figure}

We start with an arrangement $\mathcal{A}_1$ of cylindrically monotone pseudosegments depicted in Figure~\ref{fig05_NP-start}. Assume that $\gamma$ is a cylindrically monotone curve starting in $a$ and ending in $b$ that intersects every pseudosegment of $\mathcal{A}$ at most once. No such $\gamma$ can be left-oriented because it would cross some pseudosegment of $\mathcal{A}_1$ twice. 
Hence, $\gamma$ must be right-oriented. Furthermore, it cannot cross the ``right part'' of the pseudosegment $\alpha$ (the right component of $\alpha$ in the figure) since it is forced to cross the ``left part'' of $\alpha$, and so it has to go above the right part of $\alpha$ and follow the marked direction (green arrows) in Figure~\ref{fig05_NP-start}.

For every occurrence of a literal (even constant one), we will add a new pseudosegment into $\mathcal{A}_1$ starting in variable or constant gadgets and ending in clause gadgets (we will define these gadgets later). Pseudosegments corresponding to positive literals will be left-oriented, and we call them \emph{positive pseudosegments}. Pseudosegments corresponding to negative literals will be right-oriented, and we call them \emph{negative pseudosegments}. All of these new pseudosegments will be horizontal everywhere except for the regions $X_{\text{neq}}$ and $X_{\text{pos}}$. 

\begin{figure}
\begin{center}	
\includegraphics[scale=1]{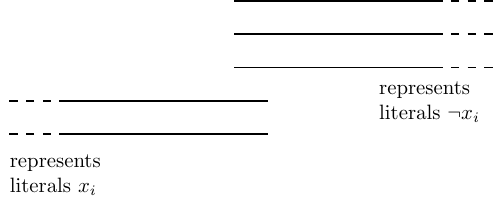}
\caption{The variable gadget of a variable $x_i$. Every pseudosegment corresponds to an occurrence of one literal. Pseudosegments corresponding to negative literals continue to the right. Pseudosegments corresponding to positive literals continue to the left.} 
\label{fig06_NP_var}
\end{center}		
\end{figure}

For each variable $x_i$ we construct the corresponding \emph{variable gadget}. It consists of a neighborhood of starting points of all pseudosegments corresponding to literals $x_i$ or $\neg x_i$, arranged so that the positive pseudosegments continue to the left and are below the negative pseudosegments, that continue to the right. Locally, the variable gadget looks as depicted in Figure~\ref{fig06_NP_var}.
Globally, we order the variable gadgets and put them into the region $G_\mathrm{var}$ depicted in Figure~\ref{fig05_NP-start}, so that the variable gadget corresponding to a variable $x_i$ is above and to the right of the variable gadgets corresponding to variables $x_1, \dots, x_{i-1}$.

For the constants $\mathtt{false_p}$ and $\mathtt{false_n}$ we construct two \emph{constant gadgets}. The constant gadget of each of the two constants consists of a neighborhood of starting points of all pseudosegments corresponding to the literals equal to this constant. The starting points of these pseudosegments are arranged on a vertical line and the pseudosegments continue to the left in the case of $\mathtt{false_p}$ or to the right in the case of $\mathtt{false_n}$. Globally, we put the constant gadget of $\mathtt{false_p}$ into the region $F_\mathrm{pos}$ as depicted in Figure~\ref{fig05_NP-start} and the constant gadget of $\mathtt{false_n}$ into $F_\mathrm{neg}$.

\begin{figure}[!htb]
\begin{center}	
\includegraphics[scale=1]{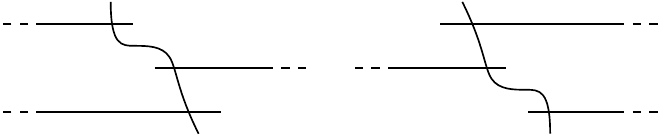}
\caption{Clause gadgets. Left: a gadget of a clause with one positive literal. Right: a gadget of a clause with two positive literals. Pseudosegments corresponding to negative literals continue to the left. Pseudosegments corresponding to positive literals continue to the right.} 
\label{fig07_NP_clause}
\end{center}		
\end{figure}

\begin{figure}
\begin{center}	
\includegraphics[scale=1]{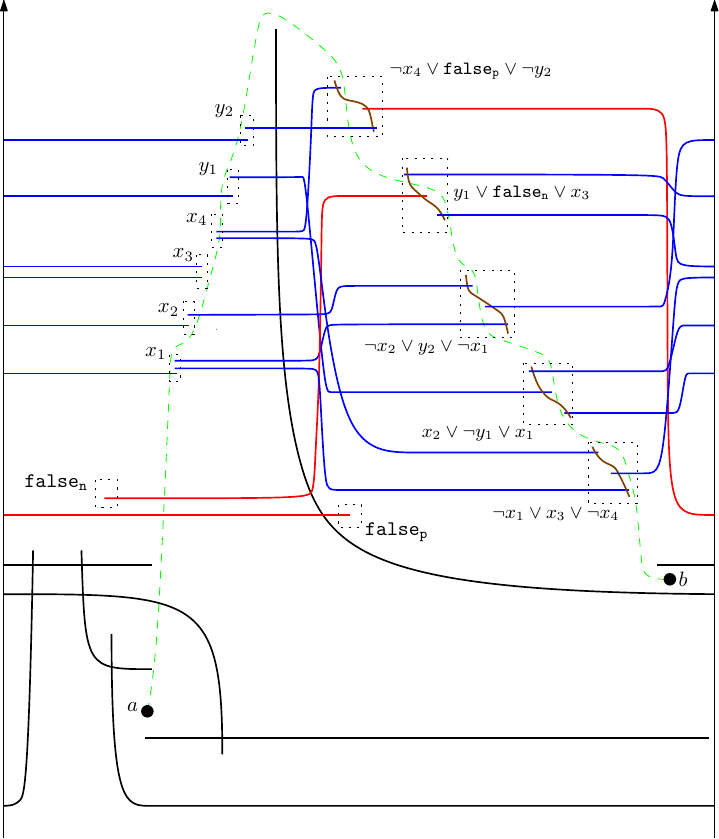}
\caption{Reduction of the satisfiability of the formula $(\neg x_1 \vee x_3 \vee \neg x_4) \wedge (x_2 \vee \neg y_1 \vee x_1) \wedge (\neg x_2 \vee y_2 \vee \neg x_1 ) \wedge (y_1 \vee \mathtt{false_n} \vee x_3) \wedge (\neg x_4 \vee \mathtt{false_p} \vee \neg y_2)$ to the problem of $(a,b)$-extendability of an arrangement of cylindrically monotone pseudosegments.\\
Pseudosegments corresponding to the constant literals are red, pseudosegments corresponding to remaining literals are blue, auxiliary pseudosegments of clause gadgets are brown, and pseudosegments of the starting arrangement $\mathcal{A}_1$ are black. A clause is satisfiable if and only if we can extend the arrangement by a new monotone pseudosegment connecting $a$ and $b$. In the figure, one such pseudosegment is drawn by the dashed green curve corresponding to the satisfying assignment $x_1, y_1, y_2 \longrightarrow  \mathtt{false}$ and $x_2, x_3, x_4\longrightarrow  \mathtt{true}$. Colors refer only to the electronic version.} 
\label{fig08_NP-end}
\end{center}		
\end{figure}

For each clause we construct a \emph{clause gadget}. It consists of a neighborhood of the three remaining endpoints of the pseudosegments corresponding to the literals in the clause together with one new auxiliary pseudosegment intersecting the three pseudosegments. There are two types of clause gadgets: one type for clauses with exactly one positive literal and the second type for clauses with exactly two positive literals. The positive pseudosegments continue to the right, and the negative pseudosegments continue to the left. In the case of a clause with one positive literal, the endpoint of the positive pseudosegment is below and to the right of the endpoint of one negative pseudosegment, and above and to the left of the endpoint of the other negative pseudosegment. The other case is analogous. Locally, the two types look as depicted in Figure~\ref{fig07_NP_clause}. 
Globally, we order the clause gadgets and put them into the region $G_\mathrm{cla}$ depicted in Figure~\ref{fig05_NP-start}, so that the clause gadget corresponding to a clause $C_i$ is below and to the right of the clause gadgets $C_1, \dots, C_{i-1}$.

It remains to connect the parts of pseudosegments from variable and constant gadgets with the parts from clause gadgets. We extend the pseudosegments from their endpoints horizontally until they reach region $X_{\text{neg}}$ or $X_{\text{pos}}$. Negative pseudosegments reach the region $X_{\text{neg}}$, positive pseudosegments reach $X_{\text{pos}}$. Then we join the corresponding parts inside $X_{\text{neg}}$ and $X_{\text{pos}}$ in such a way that each pair of pseudosegment cross at most once; see Figure~\ref{fig08_NP-end}.

Clearly, this construction is polynomial since we need only the combinatorial description of the drawing.  

We now show the correctness of the construction.

Assume that $\gamma$ is a cylindrically monotone curve joining $a$ and $b$ that intersects every pseudosegment of $\mathcal{A}$ at most once.
Split $\gamma$ into two parts at the point where it goes above $\alpha$ in between the regions $G_{\mathrm{var}}$ and $G_{\mathrm{cla}}$; denote by $\gamma_1$ the left part and by $\gamma_2$ the right part. By construction, the curve $\gamma_1$ has to intersect all pseudosegments corresponding to $\mathtt{false_p}$ and $\mathtt{false_n}$. Furthermore, by construction, for each $x_i$ the curve $\gamma_1$ has to intersect all pseudosegments corresponding to the positive literals of $x_i$ or all pseudosegments corresponding to the negative literals of $x_i$. In the former case assign $\mathtt{false}$ to $x_i$, in the latter assign $\mathtt{true}$. We claim that this is a satisfying assignment. Otherwise, some clause would not be satisfied. Consider the three pseudosegments corresponding to the literals of an unsatisfied clause (all of them are evaluated to $\mathtt{false}$). By our construction of the clause gadget, $\gamma_2$ has to cross one of these pseudosegments. But $\gamma_1$ already crossed this pseudosegment since it crossed all pseudosegments corresponding to literals evaluated to $\mathtt{false}$; a contradiction.

On the other hand, assume that there is a satisfying assignment of $\phi(x_1,\dots, x_n)$. If $x_i$ is evaluated to $\mathtt{false}$ we draw $\gamma_1$ to the left of the variable gadget corresponding to $x_i$, so that $\gamma_1$ does not intersect any pseudosegment corresponding to a negative literal $\neg x_i$. In the other case, we draw $\gamma_1$ to the right of the variable gadget corresponding to $x_i$, so that $\gamma_1$ does not intersect any pseudosegment corresponding to a positive literal $x_i$. This can clearly be done independently for all $x_i$. We draw $\gamma_2$ so that it intersects pseudosegments only in the clause gadgets. Since each clause $K$ is satisfied, there exists a pseudosegment $\rho$ corresponding to a literal of $K$ that is not intersected by $\gamma_1$. Then $\gamma_2$ can be drawn through the clause gadget of $K$ by intersecting only $\rho$ and (possibly) the auxiliary pseudosegment of the clause gadget. Hence, no pseudosegment that is crossed by $\gamma_2$ is crossed by $\gamma_1$ and so $\gamma$ is a valid $(a,b)$-extension.


\section*{Acknowledgements}
We thank the anonymous referees of earlier versions of the paper for helpful comments.



\begin{thebibliography}{99}  

\bibitem{AOV23_hamiltonian}
O. Aichholzer, J. Orthaber and B. Vogtenhuber,
Towards crossing-free Hamiltonian cycles in simple drawings of complete graphs,
{\em Comput. Geom. Topol.} {\bf 3} (2024), no. 2, Art. 5, 30 pp. 

\bibitem{AOV24_Separable}
O. Aichholzer, J. Orthaber and B. Vogtenhuber,
Separable drawings: extendability and crossing-free Hamiltonian cycles, 
{\em Proceedings of the 32nd International Symposium on Graph Drawing and Network Visualization (GD 2024)}, Leibniz International Proceedings in Informatics (LIPIcs) 320, 34:1--34:17, Schloss Dagstuhl -- Leibniz-Zentrum f\"ur Informatik (2024).

\bibitem{ABR21_pseudolines}
A. Arroyo, J. Bensmail and R. B. Richter,
Extending drawings of graphs to arrangements of pseudolines,
{\em J. Comput. Geom.} {\bf 12} (2021), no. 2, 3--24. 

\bibitem{AKPVSW23_inserting}
A. Arroyo, F. Klute, I. Parada, B. Vogtenhuber, R. Seidel and T. Wiedera,
Inserting one edge into a simple drawing is hard,
{\em Discrete Comput. Geom.} {\bf 69} (2023), no. 3, 745--770. 

\bibitem{AMRS18_Levi}
A. Arroyo, D. McQuillan, R. B. Richter and G. Salazar,
Levi's Lemma, pseudolinear drawings of $K_n$, and empty triangles,
{\em J. Graph Theory\/} {\bf 87} (2018), no. 4, 443--459. 


\bibitem{Egg73_crossing}
Roger B. Eggleton,
Crossing numbers of graphs, PhD thesis,
University of Calgary, 1973.

\bibitem{FG17_pseudoline}
S. Felsner and J. E. Goodman,
Pseudoline arrangements,
{\em Handbook of Discrete and Computational Geometry}, Third edition, Edited by Jacob E. Goodman, Joseph O'Rourke and Csaba D. T\'oth, Discrete Mathematics and its Applications (Boca Raton), CRC Press, Boca Raton, FL, 2018. ISBN: 978-1-4987-1139. Electronic version: \url{http://www.csun.edu/~ctoth/Handbook/HDCG3.html} (accessed May 2023).

\bibitem{FR13_disjoint}
R. Fulek and A. J. Ruiz-Vargas, 
Topological graphs: empty triangles and disjoint matchings,
{\em Proceedings of the twenty-ninth annual symposium on Computational geometry (SoCG '13)}, 259--266, ACM, New York, 2013.

\bibitem{Go80_proof}
J. E. Goodman,
Proof of a conjecture of Burr, Gr\"{u}nbaum, and Sloane,
{\em Discrete Math.} {\bf 32} (1980), no. 1, 27--35.

\bibitem{KPRT15_saturated}
J. Kyn\v{c}l, J. Pach, R. Radoi\v{c}i\'{c} and G. T\'{o}th,
Saturated simple and $k$-simple topological graphs,
{\em Comput. Geom.} {\bf 48} (2015), no. 4, 295--310.

\bibitem{K13_improved}
J.~Kyn\v{c}l,
Improved enumeration of simple topological graphs,
{\em Discrete Comput. Geom.} {\bf 50}(3) (2013), 727--770.

\bibitem{L26_Teilung}
F. Levi,               
Die {T}eilung der projektiven {E}bene durch {G}erade oder {P}seudogerade,
{\em Berichte Math.-Phys. Kl. S\"achs. Akad. Wiss. Leipzig} {\bf 78} (1926), 256--267.

\bibitem{LS09_circular}
M.C. Lin and J. L. Szwarcfiter,
Characterizations and recognition of circular-arc graphs and subclasses: a survey,
{\em Discrete Math.} {\bf 309} (2009), no. 18, 5618--5635. 

\bibitem{S19_Levi}
M. Schaefer,
A proof of Levi's Extension Lemma, 
\href{https://arxiv.org/abs/1910.05388v1}{arXiv:1910.05388v1} (2019).

\bibitem{SH91_sweeping}
J. Snoeyink and J. Hershberger,
Sweeping arrangements of curves,
{\em Discrete and computational geometry (New Brunswick, NJ, 1989/1990)}, 
{\em DIMACS Ser. Discrete Math. Theoret. Comput. Sci.} 6, 309--349,
Amer. Math. Soc., Providence, RI, 1991.




\end{thebibliography}
\end{document}